\documentclass[10pt]{article}
by-\headheight \advance \topmargin by-\headsep 
\usepackage{latexsym}
\usepackage{amsmath}
\usepackage{amssymb}
\usepackage{amsthm}
\usepackage{color}
\usepackage{colortbl}

\newtheorem{thm}{Theorem}[section]

\newtheorem{lemma}[thm]{Lemma}

\newtheorem{corollary}[thm]{Corollary}

\theoremstyle{definition}
\newtheorem{definition}[thm]{Definition}

\newtheorem {remark}[thm]{Remark}

\numberwithin{equation}{section}

\newcommand{\R}{\mathbb{R}}

\newcommand{\biindice}[3]%
{

\begin{array}[t]{c}
#1\\
{\scriptstyle #2}\\
{\scriptstyle #3}\\
\end{array}

}

\def\biindice#1#2#3{#1_{\textstyle{#2\atop #3}}}

\theoremstyle{plain}

\newtheorem{proposition}[thm]{Proposition}

\numberwithin{equation}{section}

\begin{document}
\title{\bf  Existence and asymptotic behavior of the least energy solutions for fractional Choquard equations with  potential  well
  \thanks{The research was supported  by the National Natural Science Foundation of China  (11671162) and the excellent doctorial dissertation
cultivation grant (No.2015YBYB022,2016YBZZ080) from Central China Normal University.
}}
\author{Lun Guo
\thanks {School of Mathematics and Statistics,  Central China Normal University, Wuhan, 430079,  P. R. China (lguo@mails.ccnu.edu.cn).}
\ \ \ \ \ \ Tingxi Hu \thanks{School of Mathematics and Statistics,
Central China Normal University, Wuhan, 430079, P. R. China (tingxihu@mails.ccnu.edu.cn).}}

\date{}
\maketitle

\begin{abstract}
In this paper, we  consider the following  nonlinear Choquard equation driven by fractional Laplacian
\begin{equation*}
(-\Delta)^{s} u+\lambda V(x)u=\big(I_{\alpha}\ast F(u)\big)f(u) \ \ \text{in}\ \ \R^{N},
\end{equation*}
where $V(x)$ is a nonnegative continuous potential function, $0<s<1$, $N> 2s$, $(N-4s)^{+}<\alpha< N$ and $\lambda$ is a positive parameter.
By  variational methods, we prove the existence of least energy solution which localizes near the bottom of potential well $int \big(V^{-1}(0)\big)$ as $\lambda$ large enough.
\end{abstract}



	\section{\label{Int}Introduction and main results}
	Given  $s\in (0,1)$, $N> 2s$ and $\alpha \in (0,N)$,  we study the following nonlinear Choquard equation driven by a fractional Laplacian operator
	\begin{equation}\label{eqs1.1}
	(-\Delta)^{s} u+V(x)u=\big(I_{\alpha}\ast F(u)\big)f(u) \ \ \text{in}\ \ \R^{N},
	\end{equation}
	where $V\in C(\R^N,\R)$ is a potential function, $F(u)=\int_{0}^{u}f(\tau)d\tau$ and $I_{\alpha}$ is the Riesz potential which is defined as
	$$
	I_{\alpha}(x):=\frac{A_{\alpha}}{|x|^{N-\alpha}}, \  \text{where} \  A_{\alpha}=\frac{\Gamma(\frac{N-\alpha}{2})}{\Gamma (\frac{\alpha}{2}\pi^{\frac{N}{2}}2^{\alpha})}\  \text{and } \Gamma ~\text{is the Gamma function}.
	$$	
The fractional Laplacian operator $(-\Delta)^s$  is defined by
\begin{equation}\label{s3-3}
(-\Delta)^s \Psi(x)=C_{N,s}P.V.
\int_{\R^{N}}
\frac{\Psi(x)-\Psi(y)}{|x-y|^{N+2s}}dy, \ \ \Psi\in \mathcal{S}(\R^N),
\end{equation}
where $P.V.$ stands for the Cauchy principal value, $C_{N,s}$ is a normalized constant,
$\mathcal{S}(\R^N)$ is the Schwartz space of rapidly decaying functions. For much more details on fractional Laplacian operator we refer the readers to \cite{DPV} and the references therein.

The fractional power of Laplacian is the infinitesimal generator of  L\'{e}vy stable diffusion process and arise in anomalous diffusion in plasma, population dynamics, geophysical fluid dynamics, flames propagation, chemical reactions in liquids and American options in finance and so on.
For interested readers we refer to \cite{Ap, FL, La1} and references therein.

	In recent years, a great attention has been  focused on the study of nonlinear equations or systems involving fractional Laplacian operators and many papers  concerned with the  existence, multiplicity, uniqueness, regularity and asymptotic behavior of solutions to fractional Schr\"{o}dinger equations are published, see for example \cite{BCD, CS6, CT,CS,CW,FW,RS,T}. We must emphasize a  remarkable work of Caffarelli and Silvestre \cite{CS}, the authors express the nonlocal operator $(-\Delta)^s$ as a Dirichlet-Neumann map for a certain elliptic boundary value problem with local differential operators defined on the upper half space. The technique of Caffarelli and Silvestre is a valid tool to deal with the equations involving fractional operators.

When $s=1$, equation \eqref{eqs1.1} is the classical nonlinear Choquard equation
	\begin{equation}\label{eqs1.3}
	-\Delta u+V(x)u=\big(I_{\alpha}\ast F(u)\big)f(u) \ \ \text{in}\ \ \R^{N}.
	\end{equation}
	Equation  \eqref{eqs1.3}  can be seen in the context of
	various physical models, such as multiple particle systems \cite{G, LS}, quantum mechanics \cite{MPT,P1,P2} and  laser beams, etc.
	
	As a special case of problem \eqref{eqs1.3} with $F(s)=s^{p}/p$, the following Choquard type equation
\begin{equation}\label{eqs1-4}
  -\Delta u+V(x)u=1/p\big(I_{\alpha}\ast |u|^{p}\big)|u|^{p-2}u \ \ \text{in}\ \ \R^{N}
\end{equation}
is studied extensively.
When $N=3$, $\alpha=2$, $p=2$ and $V\equiv1$, equation \eqref{eqs1-4} is called Choquard-Pekar equation \cite{LS,P} and also known as the Schr\"odinger-Newton equation, which was introduced by Penrose in his discussion on the selfgravitational collapse, see \cite{MPT}. In that case, By
using symmetric decreasing rearrangement inequalities, Lieb \cite{EHL} obtained existence and uniqueness of the ground state solution to equation \eqref{eqs1-4}.

It is known that  problem  \eqref{eqs1-4}  has a solution if and only if $p \in\left[\frac{N+\alpha}{N}, \frac{N+\alpha}{N-2}\right]$. If  $V(x)$ is a constant, Ma and Zhao \cite{MZ} proved that each  positive solution to equation \eqref{eqs1-4} must be radially symmetric and monotone decreasing about some fixed point under the assumption  $p\in [2,\frac{N+\alpha}{N-2})$.
Subsequently, by  variational methods,  Moroz and Van Schaftingen \cite{MS1} obtained the existence of  least energy solutions and gave some properties about the   symmetry,  regularity,  decay asymptotic behavior at infinity of the least energy solutions. In \cite{MS2}, Moroz and Van Schaftingen also obtained  a similar conclusion under the assumption of  Berestycki-Lions type nonlinearity. Equation (\ref{eqs1-4}) with lower critical exponent $p=\frac{N+\alpha}{N} $ also had been studied by Moroz and Van Schaftingen in \cite{MS3}.
If $N=3$ and $\alpha=2$,  Xiang \cite{X} obtain the uniqueness and  nondegeneracy results for the least energy solution to equation \eqref{eqs1.3} as  $p>2$ or $p$ sufficiently close to 2. When $V(x)$ is not a constant, positive solutions, sign-changing solutions, multi-bump solutions, multi-peak solutions  and normalize solutions and so on are also studied for equation \eqref{eqs1-4}, we refer the readers to \cite{ANY,CCS,CS1,LY} and references therein.

	When $s\in (0,1)$, we call equation \eqref{eqs1.1}  the fractional Choquard equation, which has also attracted a lot of interest.
 In the case $s=1/2$, problem  has been used to model the dynamics of pseudo-relativistic boson stars. Indeed, in \cite{FL}, the following equation is studied:
 $$
 \sqrt{-\Delta}u+u=\left(\frac{1}{|x|}\ast |u|^2\right)u.
 $$
 In \cite{CHKL, HL}, the authors studied the initial value problem for the boson star equation.
 Recently, d'Avenia, Siciliano and Squassina \cite{ASS} obtained some results on  existence, nonexistence, regularity, symmetry and decay properties to solutions for  equation \eqref{eqs1.1}. Chen and Liu in \cite{CL} considered a kind of non-autonomous fractional Choquard equations and obtained the existence of least energy solutions to these equations. Not too long ago, Shen, Gao and Yang in \cite{SGY} proved the existence of least energy solutions to  equation \eqref{eqs1.1} with nonlinearity satisfies the general Berestycki-Lions type assumptions.

 As far as we know,  there is no result on the existence of the least energy solution  to equation \eqref{eqs1.1} with potential well. When $s=1$, Alves,  N\'{o}brega and Yang \cite{ANY} obtained the existence of multi-bump solutions to the following equation
 \begin{equation*}
-\Delta u+(\lambda a(x)+1)u=\Big(\frac{1}{|x|^{\mu}}\ast|u|^p\Big)|u|^{p-2}u \ \ \text{in}\ \ \R^{3},
\end{equation*}
where $\mu\in (0,3)$ and $p\in (2,6-\mu)$. If the potential well $int(a^{-1}(0))$  consists of $k$ disjoint components, then  they proved that there exist at least $2^k-1$ multi-bump solutions which are concentrated at any given disjoint bounded domains of $int(a^{-1}(0))$ as the depth $\lambda$ goes to infinity.
This interesting phenomenon was first considered by Bartsch and Wang \cite{BWBW}, Ding and Tanaka \cite{DT}
for semi-linear Schr\"{o}dinger equations.
However,  some essential differences between the fractional Laplacian $(-\Delta)^{s}$ and local operator $-\Delta$ have been pointed out by Niu and Tang  \cite{NT}recently, in which they proved that the nonnegative least energy solution to fractional Schr\"{o}dinger equation cannot be trapped around only one isolated component and become arbitrary small in other components of potential well.
Due to this fact, the corresponding nonnegative least energy solution to equation \eqref{eqs1.1} must be trapped around all the domain $int \big(V^{-1}(0)\big)$, which implies we cannot obtain a similar conclusion with \cite{ANY}. Here we also want to mention that there is not any  result on the existence of multi-bump sign-changing solutions.

	Motivated by the works above, In this paper, our goal  is to investigate the existence and asymptotic behavior of least energy solutions
to  equation \eqref{eqs1.1}. Moreover, in this article, we have considered a class of Choquard type equation more general than that considered in \cite{ANY}. Also the equation we considered is  more complicated than the factional Schr\"{o}dinger equation which is considered in \cite{NT}.  Because, in our case, the  nonlinearity is much more general and  the nonlinearity, fractional Laplacian operator are both nonlocal.   In order to state our main results, we require the following assumptions on $V(x)$
	\begin{enumerate}
		\item [($V_1$)]  \ $\displaystyle V \in \mathcal{C}(\R^N, \R)$ satisfies $ V(x)\geq 0$, $\Omega:=int\big( V^{-1}(0)\big)$ is non-empty with smooth boundary and $\bar{\Omega}=V^{-1}(0)$;
		\item [($V_2$)]  \  There exists $M>0$ such that $\mu\big(\{x\in \R^N \mid V(x)\leq M\}\big)< \infty$, where $\mu$ is the Lebesgue measure;
	\end{enumerate}
and  $f(u)\in \mathcal{C}^{1}(\R,\R)$  satisfies the following assumptions
\begin{enumerate}
\item [($f_1$)]  \   $f(t)=o(t)$ as $t\rightarrow 0$;

\item [($f_2$)]  \  $\displaystyle  f(t)=o(t^p)$ as $t\rightarrow \infty$, for some $p$ satisfies $1<p <\frac{\alpha+2s}{N-2s}$;

\item [($f_3$)]  the map $t \rightarrow \frac{f(t)}{|t|}$ is nondecreasing for all $t \in \R \setminus \{0\}$.

\end{enumerate}

\begin{remark}
Conditions $(V_1)$ and $(V_2)$ were first proposed by Bartsch and Wang in \cite{BW1}. In that paper they proved the existence of a least energy
solution for $\lambda$ large enough. Furthermore,  the sequence of least energy solutions converges strongly to a least energy solution for a problem in bounded domain.
\end{remark}

\begin{remark}\label{re}
  It is important to note that from assumption $(f_3)$, we deduce that $f(t)t\geq 2F(t)$. From $(f_1)$ and $(f_3)$ we get $f(t)t>0$ with $t\neq 0$, moreover from $(f_1)$ and continuity, it follows that $f(0)=0$. Thus we get $F(t)\geq0$ for each $t\in \R$.
\end{remark}

\begin{remark}
In the present paper, $f\in \mathcal{C}^{1}(\R,\R)$ is not necessary. Suppose $f$ satisfies $(f_1)$, $(f_2)$  and Ambrosetti-Rabinowitz  condition together with $F(t)>0$, we can relax $f\in \mathcal{C}^{1}(\R,\R)$ to $f\in \mathcal{C}(\R,\R)$ and still obtain the existence of the least energy  solution to equation \eqref{eqs1.1} by constraint minimization on Nehari manifold, furthermore we show the sequence of  solutions (least energy solutions) converges to a solution (least energy solution) to the ``limit problem".
\end{remark}
Before stating our main results, we introduce some useful notations and definitions.

The fractional Sobolev space
	$H^{s}(\R^N)$ is defined as follows
	\begin{equation*}
	H^{s}(\R^{N})=\Bigg\{
	u\in L^2({\R^{N}})\ |\  \int_{\R^N}\int_{\R^N} \frac{|u(x)-u(y)|^2}
	{|x-y|^{N+2s}}dxdy<\infty
	\Bigg\}
	\end{equation*}
	equipped with the inner product
	\begin{equation*}
	\langle u,v\rangle=
	\int_{\R^N}\int_{\R^{N}} \frac{(u(x)-u(y))(v(x)-v(y))}
	{|x-y|^{N+2s}}dxdy+\int_{\R^N}uv dx
	\end{equation*}
	and the corresponding norm
	\begin{equation*}
	\|u\|_{s}=\Bigg(
	  \int_{\R^{N}} \int_{\R^{N}} \frac{|u(x)-u(y)|^2}
	{|x-y|^{N+2s}}dxdy+\int_{\R^N}|u|^2 dx
	\Bigg)^{\frac12}.
	\end{equation*}
	
The factional Laplacian operator $(-\Delta)^s$ can also be described by means of the Fourier transform, that is,
\begin{equation*}
 \mathcal{F}\big((-\Delta)^{s}u\big)(\xi)=|\xi|^{2s}\mathcal{F}(u)(\xi), \ \ \xi \in \R^N.
 \end{equation*}
where $\mathcal{F}$ denotes the Fourier transform.
It follows that, in view of Proposition 3.4 and Proposition 3.6 in \cite{DPV} that
$$
\int_{\R^3}|(-\Delta)^{\frac{s}{2}}u|^{2}dx=\int_{\R^3}|\xi|^{2s}|\mathcal{F}(u)(\xi)|^{2}d\xi
=\frac{1}{2}C(N,s)\int_{\R^N}\int_{\R^N}\frac{|u(x)-u(y)|^2}{|x-y|^{N+2s}}dxdy,
$$
for all $u\in H^{s}(\R^N)$.

It is well known that $(H^{s}(\R^{N}),\|\cdot\|_{s})$ is a uniformly convex Hilbert space and the embedding $H^{s}(\R^{N})  \hookrightarrow L^{q}(\R^{N})$ is continuous for any
	$q \in [2,2_{s}^{*}]$, where $2_{s}^{*}=\frac{2N}{N-2s}$ for $N>2s$ and $2_{s}^{*}=+\infty$ for $N\leq 2s$(See \cite{DPV}).

    To solve the problem \eqref{eqs1.1}, we will use a method due to Caffarelli and Silvestre in \cite{CS}. 
    For $u\in H^{s}$, the solution $w\in X^{s}=X^{s}(\R_{+}^{N+1})$ of
    \begin{equation}
      \begin{cases}
        -div(y^{1-2s}\nabla w)=0 ~~&\text{in}~~ \R_{+}^{N+1}, \\
        w=u  &\text{on}~~ \R^{N}\times \{0\}
      \end{cases}
    \end{equation}
is called s-harmonic extension of $u$, denoted by $w=E_{s}(u)$ and it is proved in \cite{CS} that
$$
(-\Delta)^{s}u=-\frac{1}{k_s}\lim_{y\rightarrow 0^+}y^{1-2s}\frac{\partial w}{\partial y}(x,y),
$$
where
$$
k_{s}=2^{1-2s}\frac{\Gamma(1-s)}{\Gamma (s)}.
$$
We denote   the space $X^{s}(\R_{+}^{N+1})$ as the completion of $C_{0}^{\infty}(\overline{\R_{+}^{N+1}})$ under the norm
$$
\|w\|_{X^s}:=\left(k_{s}\int_{\R_{+}^{N+1}}y^{1-2s}|\nabla w|^2 dxdy\right)^{\frac{1}{2}}.
$$
It is important to point out that the embedding $X^{s}(\R_{+}^{N+1})\hookrightarrow L^{2_{s}^{*}}(\R^N)$ is continuous(see \cite{BCD}).  Thus motivated by the  approach problem above, we will study the existence of  least energy solutions for the following problem
\begin{equation}\label{eqs1.6}
  \begin{cases}
    -div (y^{1-2s}\nabla w)=0 ~~\text{in}~~\R_{+}^{N+1}, \\
    \frac{\partial w}{\partial \nu}= -\lambda V(x)w+\big(I_{\alpha}\ast F(w)\big)f(w) ~~ \text{on}~~ \R^{N}\times \{0\},
  \end{cases}
\end{equation}
 where $\displaystyle \frac{\partial w}{\partial \nu}=-\frac{1}{k_s}\lim_{y\rightarrow 0^+}y^{1-2s}\frac{\partial w}{\partial y}(x,y)$.  From now on, we will omit the constant $k_{s}$ for convenient. Thus, if  $w\in X^{s}(\R_{+}^{N+1})$ is a solution to  problem \eqref{eqs1.6}, then the function $u(x)=w(x,0)$ will be a solution to equation \eqref{eqs1.1}.

In what follows, we define
$$
E:=\Big\{ w \in X^{s}(\R_{+}^{N+1})\mid \int_{\R^N}|w(x,0)|^2 dx< \infty\Big\}
$$
with norm
$$
\|w\|=\left( \int_{\R_{+}^{N+1}} y^{1-2s}|\nabla w|^{2}dxdy +\int_{\R^N}  |w(x,0)|^{2}dx\right)^{\frac{1}{2}}.
$$
Furthermore, $tr_{\R^N} E=H^{s}(\R^N)$, where $tr_{\R^N}E:=\{w(x,0)\mid w(x,y)\in E\}$. The embedding $E\hookrightarrow L^{q}(\R^N)$ is continuous with $2\leq q \leq 2_{s}^{*}$ and  the embedding $E\hookrightarrow L_{loc}^{q}(\R^N)$  is compact with $2\leq q < 2_{s}^{*}$  (see \cite{BCD}).

In this paper, we are looking for the least energy solution in the Hilbert space
$$
E_{\lambda}:=\Big\{ w \in X^{s}(\R_{+}^{N+1}): \int_{\R^N}\lambda V(x)|w(x,0)|^2 dx< \infty\Big\}
$$
endowed with norm
$$
\|w\|_{\lambda}=\left( \int_{\R_{+}^{N+1}} y^{1-2s}|\nabla w|^{2}dxdy +\int_{\R^N} \lambda V(x) |w(x,0)|^{2}dx\right)^{\frac{1}{2}}.
$$
Associated with \eqref{eqs1.6}, we have the energy functional $J_{\lambda}(w):E_{\lambda}\to \R$ defined by
\begin{align*}
J_{\lambda}(w):=\frac{1}{2}\int_{\R_{+}^{N+1}}y^{1-2s}|\nabla w|^{2}dxdy &+\frac{\lambda}{2} \int_{\R^N} V(x)|w(x,0)|^{2}dx \\
&-\frac{1}{2}\int_{\R^N}\Big(I_{\alpha}\ast F(w(x,0))\Big)F(w(x,0))dx.
\end{align*}
It is not difficult to find that $J_{\lambda}(w)\in \mathcal{C}^{1}(E_{\lambda},\R)$ with Gateaux derivative given by
\begin{align*}
\langle J'_{\lambda}(w), \varphi\rangle :=\int_{\R_{+}^{N+1}}y^{1-2s}\nabla w \nabla \varphi dxdy &+ \int_{\R^N} \lambda V(x)w(x,0)\varphi (x,0)dx \\
&-\int_{\R^N}\Big(I_{\alpha}\ast F(w(x,0))\Big)f(w(x,0))\varphi(x,0)dx, ~~\forall \varphi \in E_{\lambda}.
\end{align*}

\begin{definition}\label{D1.1}
    We say that $w\in E_{\lambda}$ is a weak solution to equation \eqref{eqs1.6}, if
	\begin{equation}
	\int_{\R_{+}^{N+1}}y^{1-2s}
	\nabla w\nabla\varphi dxdy\notag\\
	+\int_{\R^N}\lambda V(x)w(x,0)\varphi(x,0) dx
	=\int_{\R^N}\Big(I_{\alpha}\ast F(w(x,0))\Big)f(w(x,0))\varphi(x,0) dx
	\end{equation}
for all $\varphi \in E_{\lambda}$.
\end{definition}

In order to prove the existence of the least energy solutions to problem \eqref{eqs1.1}, we consider the following constraint minimization problem
$$
c_{\lambda}:=\inf_{w\in \mathcal{N}_{\lambda}}J_{\lambda}(w),
$$
where $\mathcal{N}_{\lambda}:=\left\{ v \in E_{\lambda}\setminus \{0\}: \langle J'_{\lambda}(w),w \rangle =0\right\}$ is the Nehari manifold.

For $\lambda$ large, the following problem
\begin{equation}\label{eqs1.7}
\begin{cases}
  (-\Delta)^{s}u=\left(\displaystyle\int_{\Omega}\frac{F(u(z))}{|x-z|^{N-\alpha}}dz\right)f(u),~~\text{in}~~\Omega, \\
  u\neq 0,~~\text{in}~~\Omega, \\
  u=0,~~\text{in}~~ \R^{N}\setminus \Omega.
\end{cases}
\end{equation}
can be seen as the limit problem of equation \eqref{eqs1.1}. In this paper, one of our aims is to prove that there exists a sequence of least energy solutions to equation \eqref{eqs1.1} converges to a least energy solution to equation \eqref{eqs1.7}. Similarly, we will study the following problem in a half space $\R_{+}^{N+1}$,
\begin{equation}\label{eqs1.8}
  \begin{cases}
    -div (y^{1-2s}\nabla w)=0 ~~\text{in}~~\R_{+}^{N+1}, \\
    \frac{\partial w}{\partial \nu}=\displaystyle \left(\int_{\Omega} \frac{F(w(z))}{|x-z|^{N-\alpha}}dz\right)f(w) ~~ \text{on}~~ \Omega\times \{0\}, \\
    w=0 ~~\text{on}~~ \R^{N} \setminus \Omega \times \{0\}.
  \end{cases}
\end{equation}
 It is obvious that if $w$ is the solution to equation \eqref{eqs1.8}, then the trace $w(x,0)$ will be a solution to equation \eqref{eqs1.7}.
In order to solve the problem \eqref{eqs1.8}, we work on a subspace $E_{0}$ of $E_{\lambda}$  defined as follows
$$
E_{0}:=\left\{ w(x,y)\in E \mid w(x,0)=0~~\text{in}~~ \R^N \setminus  \Omega \right\}.
$$
Furthermore, we define the energy functional associated with equation \eqref{eqs1.8}  by
$$
J_{0}(w):=\frac{1}{2}\int_{\R_{+}^{N+1}}y^{1-2s}|\nabla w|^{2}dxdy -\frac{1}{2}\int_{\Omega}\int_{\Omega}\frac{F(w(x,0))F(w(z,0))}{|x-z|^{N-\alpha}}dxdz.
$$
\begin{definition}\label{D1.2}
    We say that $w\in E_{0}$ is a weak solution to equation \eqref{eqs1.8}, if
	\begin{equation}
	\int_{\R_{+}^{N+1}}y^{1-2s}\nabla w \nabla \psi dxdy-\int_{\Omega}\int_{\Omega}\frac{F(w(x,0))f(w(z,0))\psi(z,0)}{|x-z|^{N-\alpha}}dxdz=0
	\end{equation}
for all $\psi \in E_{0}$.
\end{definition}
Comparing with the Nehari manifold $\mathcal{N}_{\lambda}$, we define the Nehari manifold
$$
\mathcal{N}_{0}:=\left\{ w \in E_{0}\setminus \{0\}\mid \langle J_{0}(w), w\rangle=0\right\}
$$
and
$$
c(\Omega):=\inf_{v\in \mathcal{N}_{0}} J_{0}(w)
$$
be the infimum of $J_0$ on the Nehari manifold $\mathcal{N}_{0}$.
\begin{definition}\label{D1.1}
    We call $u_{\lambda}=w_{\lambda}(x,0)$ is a least energy solution to equation \eqref{eqs1.1}, if $c_{\lambda}$ is achieved by $w_{\lambda}\in \mathcal{N}_{\lambda}$, when $w_{\lambda}$ is the critical point of $J_{\lambda}$. Similarly we say $u_0=w_0(x,0)$ is a least energy solution to equation \eqref{eqs1.7}, if $c(\Omega)$ is achieved by $w_0 \in \mathcal{N}_{0}$ which is the critical point of $J_{0}$.
\end{definition}

Then, our results can be stated as below.
\begin{thm}\label{th1.1}
	Let $N> 2s$, $\alpha \in \big((N-4s)^{+},N\big)$, where $(N-4s)^{+}=\max\{0,N-4s\}$, suppose  $(V_{1})-(V_{2})$ and $(f_1)-(f_4)$ hold. Then for $\lambda$ large enough, the problem \eqref{eqs1.1} possesses a least energy solution $u_{\lambda}(x)=w_{\lambda}(x,0)$. Furthermore, for any sequence $\lambda_{n}\rightarrow +\infty$, $\{u_{\lambda_n}(x)\}$ converges to a least energy solution to equation \eqref{eqs1.7} in $H^s (\R^N)$, up to a subsequence.
\end{thm}

Not only the least energy solution to equation \eqref{eqs1.1} has a convergent property but also any solution to equation \eqref{eqs1.1} does.  Our results on this part can be stated as follows.
\begin{thm}\label{th1.2}
	Under the same assumptions of Theorem \ref{th1.1},  let $\{u_{\lambda_n}:=w_{\lambda_n}(x,0)\}$ be a sequence of solutions to equation \eqref{eqs1.1} with $\lambda$ being replaced by $\lambda_n $ ($\lambda_n \to \infty$ as $n \to \infty$), where $w_{\lambda_n}$ denote by the s-harmonic extension of $u_{\lambda_n}$  such that $\displaystyle\limsup_{n\to \infty}
 J_{\lambda_n}(w_{\lambda_n})<\infty$. Then $u_{\lambda_n}$ converges strongly in $H^{s}(\R^N)$ to a solution to equation \eqref{eqs1.7} up to a subsequence.
\end{thm}

The paper is organized as follows. In Section 2, we give some preliminary lemmas, which are crucial in proving the compactness results. In Section 3, we consider the limit problem and give some energy estimations about $J_{\lambda}$ and $J_{0}$. In Section 4, by constraint minimization method, we prove the main results.

\section{\label{Deg} Some preliminary lemmas and compactness results }

In this Section, we first recall the well-known Hardy-Littlewood-Sobolev inequality and give some preliminary lemmas which play  important roles in showing $J_{\lambda}$ satisfies $(PS)_{c}$ condition.

\begin{lemma}\label{HLS}
	(Hardy-Littlewood-Sobolev inequality) \cite{LL}. Suppose  $\alpha \in (0,N)$, and $p$, $r>1$  with $ 1/p+1/r=1+\alpha /N$. Let $g\in L^{p}(\R^N)$, $h\in L^{r}(\R^N)$, there exists a sharp constant $C(p,\alpha,r,N)$, independent of $g$ and $h$,  such that
	\begin{equation*}
	\int_{\R^N}(I_{\alpha}\ast g)hdx  \leq C(p,\alpha,r,N)|g|_{p} |h| _{r}.
	\end{equation*}
where $|\cdot|_{p}=\left(\int_{\R^N}|u|^{p}dx\right)^{\frac{1}{p}}$.
\end{lemma}
\begin{lemma}\label{lemma2.2}
  Let $\lambda^{*}>0$ be any fixed constant, $V(x)$ satisfies $(V_1)$ and $(V_2)$.  Then the embedding $E_{\lambda} \hookrightarrow E$ is continuous  for any $\lambda>\lambda^{*}$.
\end{lemma}

\begin{proof}
  By the definitions of $E$ and $E_\lambda$, we only need to prove the following estimate
  \begin{equation}\label{eqs2.1}
  \int_{\R^N}|w(x,0)|^{2}dx \leq C\left(\int_{\R_{+}^{N+1}}y^{1-2s}|\nabla w|^{2}dxdy + \lambda \int_{\R^N}V(x)|w(x,0)|^{2}dx\right).
 \end{equation}
We define
$$
D:=\{x\in \R^N\mid V(x)\leq M\}
$$
and
$$
D^{c}:=\{x\in \R^N\mid V(x)> M\}.
$$
Thus for any function $w\in E_{\lambda}$ and $\lambda>\lambda^{*}$, we get
\begin{align}\label{eqs2.3}
  \int_{D^c}|w(x,0)|^{2}dx&\leq\frac{1}{\lambda^{*}M}\int_{D^c}\lambda V(x)|w(x,0)|^{2}dx \nonumber \\
  &\leq\frac{1}{\lambda^{*}M}\int_{\R^N}\lambda V(x)|w(x,0)|^{2}dx
\end{align}
and
\begin{align}\label{eqs2.4}
  \int_{D}|w(x,0)|^{2}dx & \leq  \mu(D)^{1-\frac{2}{2_{s}^{*}}}\left(\int_{\R^N}|w(x,0)|^{2_{s}^{*}}dx\right)^{\frac{2}{2_{s}^{*}}} \nonumber \\
  &\leq C \int_{\R_{+}^{N+1}}y^{1-2s}|\nabla w|^{2}dxdy,
\end{align}
which follows by $(V_2)$ and continuous embedding $X^{s}(\R_{+}^{N+1})\hookrightarrow L^{2_{s}^{*}}(\R^N)$. Thus by \eqref{eqs2.3} and
\eqref{eqs2.4}, we get \eqref{eqs2.1} and complete the proof.
\end{proof}

\begin{lemma}\label{lemma2.4}
Let $K_{\lambda}$ be the set of nonzero critical points for $J_{\lambda}$ with $\lambda \geq \lambda^*>0$. Then there exists a constant $\sigma_{0}>0$ independent of $\lambda$, such that
$$\|w\|_{\lambda}\geq \sigma_0,~~~~\forall w\in K_{\lambda}.$$
\end{lemma}
\begin{proof}
Suppose $w\in K_{\lambda}$, that is $w\neq 0$ and $w$ is a  critical point of $J_{\lambda}$ with $\lambda\geq \lambda^{*}>0$. Hence combining $(f_1)-(f_2)$ with Hardy-Littlewood-Sobolev inequality, we have
\begin{align*}
0=\langle J'_{\lambda}(w),w\rangle&=\int_{\R_{+}^{N+1}}y^{1-2s}|\nabla w|^{2}dxdy+ \int_{\R^N}\lambda V(x) |w(x,0)|^{2}dx \\
& \ \ \ \ \ \ \ \ \ \ \ \  \ \ \ \ \ \ \ \ \ \ \ \  \ \ \ \ \ \ \ \ \ \ -\int_{\R^N}\Big(I_{\alpha}\ast F(w(x,0))\Big)f(w(x,0))w(x,0)dx \\
&  \geq \|w\|_{\lambda}^{2}-C_1\varepsilon^{2}\|w\|^{4}-C_2 C_{\varepsilon}\|w\|^{2(p+1)}  \\
&  \geq \|w\|_{\lambda}^{2}-C_1\varepsilon^{2}\|w\|_{\lambda}^{4}-C_2 C_{\varepsilon}\|w\|_{\lambda}^{2(p+1)},
\end{align*}
where $C_1$ and $C_2$ are positive constants independent of $\lambda$  and $C_\varepsilon$ is a positive constant depend on $\varepsilon$. In the last inequality, we use the conclusion of Lemma \ref{lemma2.2}.
Thus there exists $\sigma_0>0$ such that $\|w\|_{\lambda}\geq \sigma_0$.
\end{proof}

The following lemma shows that the zero energy level of $(PS)_c$ sequence of $J_\lambda$ is isolated.
\begin{lemma}\label{lemma2.5}
Let $\{w_n\}$ be a $(PS)_{c}$ sequence for $J_{\lambda}$ with $\lambda \geq \lambda^*>0$, then $\{w_n\}$ is bounded. Furthermore, either $c=0$, or there exists a constant $c^{*}>0$ independent of $\lambda$, such that $c\geq  c^{*}$.
\end{lemma}
\begin{proof}
Suppose $\{w_n\} \subset E_{\lambda} $ is a $(PS)_{c}$ sequence for $J_{\lambda}$, that is
\begin{equation*}
J_{\lambda}(w_n)\rightarrow c \ \ \text{and}\ \ J'_{\lambda}(w_n)\rightarrow 0, \ \ \ \  \text{as} ~n\to+\infty.
\end{equation*}
Then
\begin{equation}\label{eqs2.5}
  \displaystyle J_{\lambda}(w_n)-\frac{1}{4}\langle J'_{\lambda}(w_n), w_n\rangle\leq c+ o_{n}(1)\|w_n\|_{\lambda}.
\end{equation}
Indeed,  since
\begin{equation*}
J_{\lambda}(w_n)=\frac{1}{2}\int_{\R_{+}^{N+1}}y^{1-2s}|\nabla w_n|^{2}dxdy+\frac{1}{2}\int_{\R^N}\lambda V(x)|w_{n}(x,0)|^{2}dx-\frac{1}{2}\int_{\R^N}\Big(I_{\alpha}\ast F(w_n(x,0))\Big)F(w_n(x,0))dx
\end{equation*}
and
\begin{align*}
\langle J'_{\lambda}(w_n), w_n\rangle=\int_{\R_{+}^{N+1}}y^{1-2s}|\nabla w_n|^{2}dxdy&+\int_{\R^N}\lambda V(x)|w_{n}(x,0)|^{2}dx \\
&-\int_{\R^N}\Big(I_{\alpha}\ast F(w_n(x,0))\Big)f(w_n(x,0))w_{n}(x,0)dx,
\end{align*}
thus, by the fact that  $f(t)t\geq 2F(t)\geq 0$ for each $t\in \R$ which is proved in Remark \eqref{re},  we then get
\begin{equation}\label{eqs2-6}
\displaystyle J_{\lambda}(w_n)-\frac{1}{4}\langle J'_{\lambda}(w_n), w_n\rangle
\geq \displaystyle\frac{1}{4}\int_{\R_{+}^{N+1}}y^{1-2s}|\nabla w_n|^{2}dxdy + \frac{\lambda}{4}\int_{\R^N} V(x)|w_{n}(x,0)|^{2}dx.
\end{equation}
Therefore, by \eqref{eqs2.5} and \eqref{eqs2-6}, we get
\begin{equation}\label{eqs2.7}
0\leq \frac{1}{4}\|w_n\|_{\lambda}^{2}\leq c+ o_{n}(1)\|w_n\|_{\lambda},
\end{equation}
which implies $\{w_n\}$ is bounded in $E_{\lambda}$.

By \eqref{eqs2.7}, we know $c\geq 0$. 
If $c=0$, the proof is completed. Otherwise  $c>0$, since $\left\langle J'_\lambda(w_n), w_n \right\rangle\to0$ as $n \to +\infty$, or equivalently
\begin{equation}\label{new label 2}
    \|w_n\|^2_\lambda=\int_{\R^N}\Big(I_{\alpha}\ast F(w_n(x,0))\Big)f(w_n(x,0))w_n(x,0)dx+o_n(1).
\end{equation}
Using the Hardy-Littlewood-Sobolev inequality, continuous embeddings $E_{\lambda}\hookrightarrow E\hookrightarrow L^{q}(\R^N)$ with $q\in[2,2_{s}^{*}]$ together with assumptions $(f_1)$ and $(f_2)$, we have
\begin{equation}\label{add1}
    \|w_n\|^2_\lambda\leq C_3\max\left\{ \|w_n\|_\lambda^4,~~ \|w_n\|_\lambda^{2p+2}\right\},
\end{equation}
for some $C_3>0$ which is independent of $\lambda$. 

Thus
by \eqref{add1}, there exists  $\delta_1>0$ such that $\displaystyle\liminf_{n\rightarrow \infty}\|w_n\|_\lambda\geq\delta_1$. Let $c^\ast=\delta_1^2/ 4>0$ which is independent of $\lambda$, hence by \eqref{eqs2.7} we have
\begin{equation*}
    c=\lim_{n\to+\infty}J_\lambda(w_n)\geq \lim_{n\to +\infty}\frac{1}{4}\|w_n\|_\lambda^2\geq c^\ast.
\end{equation*}
\end{proof}

\begin{lemma}\label{lemma2.6}
Let $\{w_n\}$ be a $(PS)_c$ sequence for $J_{\lambda}$ with $\lambda\geq \lambda^{*}>0$ and $c>0$. Then there exists a constant $\delta_{2}>0$ independent of $\lambda$, such that
\begin{equation*}
\liminf_{n\rightarrow +\infty}\Big\{\int_{\R^N}|f(w_n(x,0))w_n(x,0)|^{\frac{2N}{N+\alpha}}dx\Big\}^{\frac{N+\alpha}{N}}\geq \delta_2 c.
\end{equation*}
\end{lemma}
\begin{proof}
Before proving the lemma, we  first point out the fact that  $f(t)t\geq 2F(t)\geq 0$  for all $t \in \R$.  Since  $\{w_n\}$ is a $(PS)_c$ sequence for $J_{\lambda}$, then by Hardy-Littlewood-Sobolev inequality, we get
\begin{equation}
\begin{array}{ll}
\displaystyle c&=\displaystyle\lim_{n\rightarrow +\infty}\left(J_{\lambda}(w_n)-\frac{1}{2}\langle J'_{\lambda}(w_n), w_n\rangle\right) \\
&=\displaystyle\frac{1}{2}\lim_{n\rightarrow +\infty}\int_{\R^N}\Big(I_{\alpha}\ast F(w_n(x,0))\Big)\Big(f\big(w_n(x,0)\big)w_n(x,0)-F\big(w_n(x,0)\big)\Big)dx \\
&\leq\displaystyle \frac{1}{4}\lim_{n\rightarrow+\infty}\int_{\R^N}\Big[I_{\alpha}\ast \Big(f\big(w_n(x,0)\big)w_n(x,0)\Big)\Big] f\big(w_n(x,0)\big)w_n(x,0)dx \\
&\leq\displaystyle C_4 \liminf_{n\rightarrow+\infty}\Big\{\int_{\R^N}|f(w_n(x,0))w_n(x,0)|^{\frac{2N}{N+\alpha}}dx\Big\}^{\frac{N+\alpha}{N}}.
\end{array}
\end{equation}
Setting $\displaystyle\delta_2=1/{C_4}$, we then get $\displaystyle\liminf_{n\rightarrow +\infty}\Big\{\int_{\R^N}|f(w_n(x,0))w_n(x,0)|^{\frac{2N}{N+\alpha}}dx\Big\}^{\frac{N+\alpha}{N}}\geq c/{C_4}=\delta_2 c$.
\end{proof}

\begin{lemma}\label{lemma2.7}
Let $\bar{C}>0$  be fixed and independent of $\lambda$, $\{w_n\}$ be a $(PS)_c$ sequence for $J_{\lambda}$ with $c\in [0,\bar{C}]$. Given $\varepsilon>0$, there exist $\Lambda_{\varepsilon}=\Lambda(\varepsilon)$ and $R_{\varepsilon}=R(\varepsilon)$ such that
\begin{equation}
\limsup_{n\rightarrow +\infty}\Big\{\int_{\R^N\setminus B_{R_\varepsilon}(0)}|f(w_n(x,0))w_n(x,0)|^{\frac{2N}{N+\alpha}}dx\Big\}^{\frac{N+\alpha}{N}}\leq \varepsilon, \ \ \ \ \forall ~\lambda\geq \Lambda_{\varepsilon}.
\end{equation}
\end{lemma}
\begin{proof}
For $R>0$, we define
\begin{equation*}
A(R)=\{x\in \R^N \mid |x|>R \ \ \text{and}\ \ V(x)\geq M\},
\end{equation*}
and
\begin{equation*}
B(R)=\{x\in \R^N\mid |x|>R \ \ \text{and}\ \ V(x)< M\}.
\end{equation*}
Hence,  with a direct calculation, we get
\begin{equation}\label{eqs4.22}
\begin{array}{ll}
\displaystyle\int_{A(R)}|w_{n}(x,0)|^{2}dx &\displaystyle\leq \frac{1}{\lambda M}\int_{\R^N}\lambda V(x)|w_{n}(x,0)|^{2}dx \\
&\leq  \displaystyle\frac{1}{\lambda M} \|w_n\|_{\lambda}^{2}   \\
& \leq \displaystyle \frac{1}{\lambda M} \big(4c+o_{n}(1)\|w_n\|_{\lambda}\big) \\
& \leq \displaystyle \frac{1}{\lambda M}\big(4\bar{C}+o_{n}(1)\|w_n\|_{\lambda}\big),
\end{array}
\end{equation}
where in the third inequality, we have used \eqref{eqs2.7}.

Since $\bar{C}$ is independent of $\lambda$, then by \eqref{eqs4.22},  there exists some $\Lambda_{\varepsilon}>0$, such that
\begin{equation}\label{eqs4.23}
\limsup_{n\rightarrow +\infty}\int_{A(R)}|w_{n}(x,0)|^{2}dx <\frac{\varepsilon}{4}, \ \ \forall ~\lambda \geq \Lambda_{\varepsilon}.
\end{equation}
By using the H\"{o}lder inequality and continuous embeddings $E_{\lambda}\hookrightarrow E\hookrightarrow L^{q}(\R^N)$ with $q\in[2,2_{s}^{*}]$, we have
\begin{equation}\label{eqs4.24}
\int_{B(R)}|w_{n}(x,0)|^{2}dx \leq C \|w_n\|_{\lambda}^{2} \cdot \mu (B(R))^{\frac{2s}{N}}\leq 4\bar{C}C \cdot \mu(B(R))^{\frac{2s}{N}}+o_{n}(1).
\end{equation}
Furthermore by $(V_2)$, we know that $\mu(B(R))\rightarrow 0$ as $R\rightarrow +\infty$. Thus we choose $R_\varepsilon:=R(\varepsilon)$ large enough such that
\begin{equation}\label{eqs4.25}
\limsup_{n\rightarrow+\infty}\int_{B(R_\varepsilon)}|w_{n}(x,0)|^{2}dx< \frac{\varepsilon}{4}.
\end{equation}
Setting $\lambda\geq \Lambda_{\varepsilon}$, $R =R_{\varepsilon}$ and combining \eqref{eqs4.23} with \eqref{eqs4.25}, we  obtain
\begin{equation}\label{eqs4.26}
\limsup_{n\rightarrow+\infty}\int_{\R^N\setminus B_{R_\varepsilon}(0)}|w_{n}(x,0)|^{2}dx <\frac{\varepsilon}{4}+\frac{\varepsilon}{4}=\frac{\varepsilon}{2}.
\end{equation}
Since $\{w_n\}$ is a $(PS)_c$ sequence, hence by Lemma \ref{lemma2.5} we know that $\{w_n\}$ must be bounded in $E_\lambda$. By interpolation inequality and \eqref{eqs4.26} we have
\begin{equation*}
\displaystyle\limsup_{n\rightarrow +\infty}\int_{\R^N\setminus B_{R_\varepsilon}(0)}|w_n(x,0)|^{\frac{4N}{N+\alpha}}dx<\frac{\varepsilon}{2}, \ \ \ \ \forall\lambda>\Lambda_{\varepsilon}
\end{equation*}
and
\begin{equation*}
\displaystyle\limsup_{n\rightarrow +\infty}\int_{\R^N\setminus B_{R_\varepsilon}(0)}|w_n(x,0)|^{\frac{2N(p+1)}{N+\alpha}}dx<\frac{\varepsilon}{2}, \ \ \ \ \forall \lambda>\Lambda_\varepsilon.
\end{equation*}
Thus we get
\begin{equation*}
\begin{array}{ll}
& \ \ \ \ \displaystyle\limsup_{n\rightarrow +\infty}\Big\{\int_{\R^N\setminus B_{R_\varepsilon}(0)}|f(w_n(x,0))w_n(x,0)|^{\frac{2N}{N+\alpha}}dx\Big\}^{\frac{N+\alpha}{N}}  \\
&\leq\displaystyle C\left(\limsup_{n\rightarrow +\infty}\Big\{\int_{\R^N\setminus B_{R_\varepsilon}(0)}| w_n(x,0)|^{\frac{4N}{N+\alpha}}dx\Big\}^{\frac{N+\alpha}{N}}+\limsup_{n\rightarrow +\infty}\Big\{\int_{\R^3\setminus B_{R_\varepsilon}(0)}| w_n(x,0)|^{\frac{2N(p+1)}{N+\alpha}}dx\Big\}^{\frac{N+\alpha}{N}}\right) \\
&\leq \varepsilon,
\end{array}
\end{equation*}
which follows by the fact $f(t)t\leq \varepsilon t^2+C_\varepsilon t^{p+1}$.
\end{proof}
The following lemma is a Br\'ezis-Lieb type Lemma for Choquard type equation.
\begin{lemma}\label{lemma2.8}
Let  $\{w_n\}$ be a $(PS)_{c}$ sequence for $J_{\lambda}$. If $w_n\rightharpoonup w$ in $E_{\lambda}$, then
\begin{equation}\label{eqs2.17}
\lim_{n\rightarrow \infty}\big(J_{\lambda}(w_n)-J_{\lambda}(v_n)\big)=J_{\lambda}(w),
\end{equation}
\begin{equation}\label{eqs2.18}
\lim_{n\rightarrow \infty}\big(J'_{\lambda}(w_n)-J'_{\lambda}(v_n)\big)=J'_{\lambda}(w).
\end{equation}
where $v_n=w_n-w$. Furthermore, $w$ is a weak solution to equation \eqref{eqs1.6} and $\{v_n\}$ is a $(PS)_{c-J_{\lambda}(w)}$ sequence.
\end{lemma}
\begin{proof}
We only give the proof of \eqref{eqs2.17}, with a similar argument, \eqref{eqs2.18} can  also be proved.
In order to complete the proof, it is sufficient to prove a Br\'ezis-Lieb type lemma for the nonlocal term, more precisely,
\begin{equation}
\begin{split}
\lim_{n\rightarrow +\infty}&\Big(\int_{\R^N}\Big(I_{\alpha}\ast F(w_n(x,0))\Big)F(w_n(x,0))dx-\int_{\R^N}\Big(I_{\alpha}\ast F(v_n(x,0))\Big)F(v_n(x,0))dx\Big)\\
=&\int_{\R^N}\Big(I_{\alpha}\ast F(w(x,0))\Big)F(w(x,0))dx.
\end{split}
\end{equation}
Define
\begin{align*}
 \mathcal{G}=\displaystyle\int_{\R^N}\Big(I_{\alpha}\ast F(w_n(x,0))\Big)F(w_n(x,0))dx &-\int_{\R^N}\Big(I_{\alpha}\ast F(v_n(x,0))\Big)F(v_n(x,0))dx \\
 & -\int_{\R^N}\Big(I_{\alpha}\ast F(w(x,0))\Big)F(w(x,0))dx.
\end{align*}
With a direct computation, we obtain that
\begin{align}
\mathcal{G}&= \displaystyle\int_{\R^N}\Big(I_{\alpha} \ast F(w_n(x,0))\Big) \Big(F(w_{n}(x,0))-F(v_{n}(x,0))-F(w(x,0))\Big)dx  \label{eqs2.19}\\
&\ \ \ \ \ \displaystyle+ \int_{\R^N}\Big(I_{\alpha} \ast F(v_n(x,0))\Big)\Big (F(w_{n}(x,0))-F(v_{n}(x,0))-F(w(x,0))\Big) dx \label{eqs2.20}\\
&\ \ \ \ \ \displaystyle+ \int_{\R^N}\Big(I_{\alpha} \ast F(w(x,0))\Big)\Big (F(w_{n}(x,0))-F(v_{n}(x,0))-F(w(x,0))\Big) dx \label{eqs2.21}  \\
& \ \ \ \ \  \displaystyle +2\int_{\R^N}\Big(I_{\alpha}\ast F(w(x,0))\Big)F(v_n(x,0))dx  \label{eqs2.22}.
\end{align}
Applying the Hardy-Littlewood-Sobolev inequality to the nonlocal terms in (\ref{eqs2.19})--(\ref{eqs2.21}), one has
\begin{align*}
 \mathcal{G } &\leq C | F(w_n(x,0))|_{\frac{2N}{N+\alpha}} |F(w_{n}(x,0))-F(v_{n}(x,0))-F(w(x,0))|_{\frac{2N}{N+\alpha}}\\
    &\ \ \ \ + C | F(v_n(x,0))|_{\frac{2N}{N+\alpha}} |F(w_{n}(x,0))-F(v_{n}(x,0))-F(w(x,0))|_{\frac{2N}{N+\alpha}}\\
    &\ \ \ \ + C | F(w(x,0))|_{\frac{2N}{N+\alpha}}|F(w_{n}(x,0))-F(v_{n}(x,0))-F(w(x,0))|_{\frac{2N}{N+\alpha}}\\
     & \ \ \ \  \displaystyle +2\int_{\R^N}\Big(I_{\alpha}\ast F(w(x,0))\Big)F(v_n(x,0))dx,
\end{align*}
for some $C>0$.

Without loss of generality, we assume $w_n\rightharpoonup w$ in $E_\lambda$ up to a subsequence, then $\{w_n\}$ is bounded in $E_{\lambda}$.
Hence under assumptions $(f_1)$ and $(f_2)$, we get
\begin{equation*}
    \begin{split}
      & | F(w_n(x,0))|_{\frac{2N}{N+\alpha}}+| F(v_n(x,0))|_{\frac{2N}{N+\alpha}}+| F(w(x,0))|_{\frac{2N}{N+\alpha}}\\
\leq & \varepsilon\left(|w_n(x,0)|^2_{\frac{4N}{N+\alpha}}+|w(x,0)|^2_{\frac{4N}{N+\alpha}}+|v_n(x,0)|^2_{\frac{4N}{N+\alpha}}\right)\\
& +C_\varepsilon \left(|w_n(x,0)|^{p+1}_{\frac{2N(p+1)}{N+\alpha}}+|w(x,0)|^{p+1}_{\frac{2N(p+1)}{N+\alpha}}+|v_n(x,0)|^{p+1}_{\frac{2N(p+1)}{N+\alpha}}\right)\\
\leq &C.
    \end{split}
\end{equation*}

We claim that $F(w_{n}(x,0))-F(v_{n}(x,0))-F(w(x,0))\to 0$ strongly in $L^{\frac{2N}{N+\alpha}}(\R^N)$ as $n\to +\infty$. In fact, as $w_n\rightharpoonup w$  in $E_{\lambda}$, it follows that  $w_n(x,0)\to w(x,0)$ strongly  in $L_{loc}^{q}(\R^N)$ for any $q\in \left[2,2_{s}^{*}\right)$ and $w_n(x,0)\to w(x,0)$ a.e. in $\R^N$.
Thus we have $F(w_n(x,0))\rightarrow F(w(x,0))$ strongly in $L^{\frac{2N}{N+\alpha}}(B_{R_\varepsilon}(0))$  and
$F(v_n(x,0))\rightarrow 0$ strongly in $L^{\frac{2N}{N+\alpha}}(B_{R_\varepsilon}(0))$, moreover
\begin{equation}\label{eqs4.13}
\begin{array}{ll}
&\ \ \ \ \displaystyle\Big\{\int_{B_{R_\varepsilon}(0)}|F(w_n(x,0))-F(v_{n}(x,0))-F(w(x,0))|^{\frac{2N}{N+\alpha}}dx\Big\}^{\frac{N+\alpha}{2N}} \\
&\leq \displaystyle\Big\{\int_{B_{R_\varepsilon}(0)}|F(w_n(x,0))-F(w(x,0))|^{\frac{2N}{N+\alpha}}dx\Big\}^{\frac{N+\alpha}{2N}}
+\displaystyle\Big\{\int_{B_{R_\varepsilon}(0)}|F(v_n(x,0))|^{\frac{2N}{N+\alpha}}dx\Big\}^{\frac{N+\alpha}{2N}} \\
&=o_{n}(1).
\end{array}
\end{equation}
For some $ \bar\theta\in (0,1)$, we have
\begin{align}\label{eqs2.6}
&\displaystyle\ \ \ \ \int_{\R^N \setminus B_{R_\varepsilon}(0)}|F(w_n(x,0))-F(v_{n}(x,0))-F(w(x,0))|^{\frac{2N}{N+\alpha}}dx\nonumber \\
&\leq \displaystyle\int_{\R^N \setminus B_{R_\varepsilon}(0)}|F(w_n(x,0))-F(v_n(x,0))|^{\frac{2N}{N+\alpha}}dx
+\displaystyle\int_{\R^N \setminus B_{R_\varepsilon}(0)}|F(w(x,0))|^{\frac{2N}{N+\alpha}}dx\nonumber\\
&= \displaystyle\int_{\R^N \setminus B_{R_\varepsilon}(0)}|f\big(v_n(x,0)+\bar\theta w(x,0)\big)w(x,0)|^{\frac{2N}{N+\alpha}}dx+o_{R_\varepsilon}(1)\quad \\
&\leq\displaystyle C\int_{\R^N \setminus B_{R_\varepsilon}(0)} |v_n(x,0)|^{\frac{2N}{N+\alpha}}|w(x,0)|^\frac{2N}{N+\alpha}+|w(x,0)|^{\frac{4N}{N+\alpha}}dx   \nonumber\\
& \ \ \ \
+C\int_{\R^N\setminus B_{R_\varepsilon}(0)}|v_n(x,0)|^{\frac{2Np}{N+\alpha}}|w(x,0)|^{\frac{2N}{N+\alpha}}+|w(x,0)|^{\frac{2N(p+1)}{N+\alpha}}dx + o_{R_\varepsilon}(1). \nonumber
\end{align}
Since $N> 2s$ and $\alpha\in \big((N-4s)^{+},N\big)$, there exist some $p^{*}$, $q^{*}$ with $1<\frac{N+\alpha}{N}\leq p^* $ and $1<\frac{N+\alpha}{2s+\alpha}\leq q^*$ such that
\begin{align}\label{2.27}
&\ \ \ \ \displaystyle\int_{\R^N \setminus B_{R_\varepsilon}(0)} |v_n(x,0)|^\frac{2N}{N+\alpha}|w(x,0)|^\frac{2N}{N+\alpha}dx \nonumber \\
&\leq \left\{\int_{\R^N} |v_n(x,0)|^\frac{2Np^*}{N+\alpha}dx\right\}^{\frac{1}{p^*}}\left\{\int_{\R^N\setminus B_{R_\varepsilon}(0)}|w(x,0)|^{\frac{2Np^{*}}{(N+\alpha) (p^{*}-1)}}dx\right\}^{1-\frac{1}{p^*}}\rightarrow 0,
\end{align}
and
\begin{align}\label{2.28}
&\ \ \ \  \displaystyle\int_{\R^N \setminus B_{R_\varepsilon}(0)} |v_n(x,0)|^\frac{2Np}{N+\alpha}|w(x,0)|^{\frac{2N}{N+\alpha}}dx  \nonumber \\
&\leq \left\{\int_{\R^N} |v_n(x,0)|^\frac{2Npq^*}{N+\alpha}dx\right\}^{\frac{1}{q^*}}\left\{\int_{\R^N\setminus B_{R_\varepsilon}(0)}|w(x,0)|^{\frac{2Nq^{*}}{(N+\alpha)(q^{*}-1)}}dx\right\}^{1-\frac{1}{q^*}} \rightarrow 0.
\end{align}
Thus substituting (\ref{2.27}) and (\ref{2.28}) into \eqref{eqs2.6} and taking the limit $n\to +\infty$ firstly, then $R_\varepsilon\to +\infty$ subsequently, we  obtain
\begin{equation}\label{eqsaddadd}
 \int_{\R^N \setminus B_{R_\varepsilon}(0)}|F(w_n(x,0))-F(v_{n}(x,0))-F(w(x,0))|^{\frac{2N}{N+\alpha}}dx\to 0.
\end{equation}
It follows from \eqref{eqs4.13} and \eqref{eqsaddadd} that
\begin{equation}
\int_{\R^N}|F(w_n(x,0))-F(w(x,0))-F(v_n(x,0))|^{\frac{2N}{N+\alpha}}dx \rightarrow 0  ~~\text{as}~~ n\to \infty.
\end{equation}
%
Before completing the proof, we still need to prove
\begin{equation}\label{new label 1}
   \displaystyle\int_{\R^N}\Big(I_{\alpha}\ast F(w(x,0))\Big)F(v_n(x,0))dx \to 0.
\end{equation}
By the facts that $v_{n}\rightharpoonup 0$ in $E_{\lambda}$ and $F(v_n(x,0))$ is bounded in $L^{\frac{2N}{N+\alpha}}(\R^N)$, we then assert
\begin{equation}\label{a}
F(v_n(x,0))\rightharpoonup 0~~ \text{in}~~ L^{\frac{2N}{N+\alpha}}(\R^N).
\end{equation}
As $F(w(x,0))\in L^{\frac{2N}{N+\alpha}}(\R^N)$, thus by the Hardy-Littlewood-Sobolev inequality,
\begin{equation}\label{b}
  I_{\alpha}\ast F(w(x,0))\in L^{\frac{2N}{N-\alpha}}(\R^N).
\end{equation}
By \eqref{a} and \eqref{b}, we then prove \eqref{new label 1} and complete the  proof.
\end{proof}
Now, we prove the following compactness result.
\begin{proposition}\label{pro2.1}
  Suppose that $(V_1)-(V_2)$ and $(f_1)-(f_3)$ hold. Then for any $\widetilde{C}>0$, there exists $\Lambda_0>0$ such that $J_{\lambda}$ satisfies the $(PS)_c$ condition for each $\lambda\geq \Lambda_0$ and $c\leq \widetilde{C}$.
\end{proposition}
\begin{proof}

  Let $\{w_n\}$ be a $(PS)_{c}$ sequence of $J_\lambda$, where $\lambda>\Lambda_0$ and $c\leq \widetilde{C}$, then as a direct consequence of Lemma \ref{lemma2.5}, we know $\{w_n\}$ is bounded in $E_{\lambda}$. Without loss of generality,  there exists some $w\in E_\lambda$ such that  $w_n\rightharpoonup w$ in $E_\lambda$ up to a subsequence, moreover $v_n=w_n-w$ is a $(PS)_{c-J_{\lambda}(w)}$ sequence which follows by Lemma \ref{lemma2.8}.

We claim that $d:=c-J_{\lambda}(w)=0$. If not, we suppose that $d>0$. It follows from the
 Lemmas \ref{lemma2.5} and \ref{lemma2.6} that there exists some $c_{\ast}>0$ satisfies $ d\geq c_{*}$  and
\begin{equation}\label{eqs2.30}
\displaystyle\liminf_{n\rightarrow +\infty}\Big\{\int_{\R^N}|f(v_n(x,0))v_n(x,0)|^{\frac{2N}{N+\alpha}}dx\Big\}^{\frac{N+\alpha}{N}}\geq \delta_2 d\geq \delta_2 c_{*}>0.
\end{equation}
Let $\varepsilon\in (0,\delta_2c_{*}/2)$ and  $\Lambda_0:=\Lambda_\varepsilon$, by Lemma \ref{lemma2.7},
 we then deduce that
\begin{equation}\label{eqs2.31}
\limsup_{n\rightarrow+\infty}\Big\{\int_{\R^N\setminus B_{R_\varepsilon}(0)}|f(v_n(x,0))v_n(x,0)|^{\frac{2N}{N+\alpha}}dx\Big\}^{\frac{N+\alpha}{N}}\leq \frac{1}{2}\delta_2c_{*} \ \  \forall \lambda \geq \Lambda_0,
\end{equation}
where $R_{\varepsilon}$  is given in Lemma \ref{lemma2.7}.
From \eqref{eqs2.30} and \eqref{eqs2.31}, we  get
\begin{equation}\label{eqs2.32}
\liminf_{n\rightarrow+\infty}\Big\{\int_{ B_{R_\varepsilon}(0)}|f(v_n(x,0))v_n(x,0)|^{\frac{2N}{N+\alpha}}dx\Big\}^{\frac{N+\alpha}{N}}\geq \frac{1}{2}\delta_2c_{*}>0.
\end{equation}
However, since $E_{\lambda}$ embedded into $L_{loc}^{q}(\R^N)$ compactly for $2\leq q <\frac{2N}{N-2s}$, thus
\begin{equation}
\liminf_{n\rightarrow +\infty}\Big\{\int_{ B_{R_\varepsilon}(0)}|f(v_n(x,0))v_n(x,0)|^{\frac{2N}{N+\alpha}}dx\Big\}^{\frac{N+\alpha}{N}}=0,
\end{equation}
which contradicts to \eqref{eqs2.32}. So $d=0$ and $\{v_n\}$ is a $(PS)_{0}$ sequence. Therefore by  \eqref{eqs2.7} we deduce that $v_{n}\rightarrow 0$ in $E_{\lambda}$, which implies that  $J_{\lambda}$ satisfies $(PS)_c$ condition for $c\in [0,\widetilde{{C}}]$ provided $\lambda>\Lambda_0$.
\end{proof}

\section{\label{Deg} Limit problem }

Recall that the following problem can be seen as the limit problem of equation \eqref{eqs1.6}
\begin{equation}\label{eqs3.1}
  \begin{cases}
    -div (y^{1-2s}\nabla w)=0 ~~\text{in}~~\R_{+}^{N+1}, \\
    \frac{\partial w}{\partial \nu}= \displaystyle\left(\int_{\Omega}\frac{F(w(z))}{|x-z|^{N-\alpha}}dz\right)f(w) ~~ \text{on}~~ \Omega\times \{0\}, \\
    v=0 ~~\text{on}~~ \R^{N} \setminus \Omega \times \{0\}
  \end{cases}
\end{equation}
and the corresponding functional of equation \eqref{eqs3.1} is defined by
$$
J_{0}(w):=\frac{1}{2}\int_{\R_{+}^{N+1}}y^{1-2s}|\nabla w|^{2}dxdy -\frac{1}{2}\int_{\Omega}\int_{\Omega}\frac{F(w(x,0))f(w(z,0))w(z,0)}{|x-z|^{N-\alpha}}dxdz, ~~ \forall  w\in E_0.
$$
As defined in Section 1,
$$
c(\Omega):=\inf_{w\in \mathcal{N}_{0}} J_{0}(w)
$$
 is the infimum of $J_0$ on the Nehari manifold $\mathcal{N}_{0}$. In the following part, we want to prove $c(\Omega)$ is achieved. To show that, we firstly give an embedding lemma which is standard.
\begin{lemma}\label{lemma3.1}
   The embedding $tr_{\Omega} E_0 \hookrightarrow L^{q}(\Omega)$ is compact for $q\in [2, 2_{s}^{*})$.
\end{lemma}
\begin{proof}
  The proof is trivial. Since $tr_{\Omega}E_{0}\subset H^{s}(\Omega)$ and the embedding $H^{s}(\Omega)\hookrightarrow L^{q}(\Omega)$ is compact for $q\in[2, 2_{s}^{*})$, hence  the embedding $tr_{\Omega} E_0 \hookrightarrow L^{q}(\Omega)$ is compact for $q\in [2, 2_{s}^{*})$.
\end{proof}
\begin{lemma}
  The infimum $c(\Omega)$ is achieved by a function $w_0\in \mathcal{N}_0$ which is a least energy solution to \eqref{eqs3.1}.
\end{lemma}
\begin{proof}
  By Ekeland's Variational Principle, there exist a $(PS)_{c(\Omega)}$ sequence $\{w_{n}\}\subset E_0$ such that
  $$
  J_0(w_n)\rightarrow c(\Omega) \ \ \text{and} \ \ J'_{0}(w_n)\rightarrow 0.
  $$
  Thus we have
  \begin{align*}
    c(\Omega)+o_{n}(1)\|w_n\|&\geq J_{0}(w_n)-\frac{1}{4}\langle J'_{0}(w_n), w_n\rangle \\
    &\geq \displaystyle\frac{1}{4}\int_{\R_{+}^{N+1}}y^{1-2s}|\nabla w_n|^{2}dxdy + \frac{1}{4}\int_{\Omega}\int_{\Omega}\frac{F(w_n(x,0))\Big(f(w_n(z,0))w_n(z,0)-2F(w_n(z,0))\Big)}{|x-z|^{N-\alpha}}dxdz \\
    &\geq \frac{1}{4}\int_{\R_{+}^{N+1}}y^{1-2s}|\nabla w_n|^{2}dxdy \\
    & \geq \frac{1}{8} \int_{\R_{+}^{N+1}}y^{1-2s}|\nabla w_n|^{2}dxdy + \frac{C_5}{8}\Big(\int_{\Omega}|w_n(x,0)|^{2_{s}^{*}}dx\Big)^\frac{2}{2_{s}^{*}} \\
    & \geq \frac{1}{8} \int_{\R_{+}^{N+1}}y^{1-2s}|\nabla w_n|^{2}dxdy + \frac{C_6}{8}\int_{\Omega}|w_n(x,0)|^{2}dx\\
    &\geq C \|w_n\|^2,
  \end{align*}
 where we choose $C\leq \min \{1/8, {C_6}/8\}$. Thus $\{w_n\}$ is bounded in $E_0$. Furthermore, there exists a $w_0 \in E_0$ such that $w_{n}\rightharpoonup w_0$ in $E_0$ up to a subsequence, furthermore by Lemma \ref{lemma3.1}, $w_n \to w_0$ in $L^{q}(\Omega)$ with $q\in [2,2_{s}^{*})$. Then
 \begin{align*}
 \|w_n-w_0\|^2&=\int_{\R_{+}^{N+1}}y^{1-2s}|\nabla w_n-\nabla w_0|^{2}dxdy + \int_{\Omega} |w_n(x,0)-w(x,0)|^{2}dx \\
  &=\int_{\R_{+}^{N+1}}y^{1-2s}|\nabla w_n|^{2}dxdy+\int_{\Omega} |w_n(x,0)|^{2}dx \\
  & \ \ \ \ \ \ \ \ \ \ \ \ \ \ \ \ \ \ \ \ \ \ \ \ \ \ \ \ \ \ \ \ \ \ \ \
  -\int_{\R_{+}^{N+1}}y^{1-2s}|\nabla w|^{2}dxdy-\int_{\Omega}|w(x,0)|^{2}dx +o_{n}(1) \\
  &= \int_{\Omega}\int_{\Omega}\frac{F(w_n(x,0))f(w_n(z,0))w_n(z,0)}{|x-z|^{N-\alpha}}dxdz -\int_{\Omega}\int_{\Omega}\frac{F(w(x,0))f(w(z,0))w(z,0)}{|x-z|^{N-\alpha}}dxdz +o_{n}(1) \\
  & \to 0 \ \ \text{as}~~ n \to \infty.
\end{align*}
 Thus $w_{n} \to w_0$ strongly in $E_0$, furthermore $J_0(w_0)=c(\Omega)$ and $J'_{0}(w_0)=0$ . Therefore $w_0$ is a least energy solution to equation \eqref{eqs3.1} and we complete the proof.
\end{proof}
\begin{remark}
  If $f$ is odd and satisfies $f(t)\geq 0$ for $t\in [0,+\infty)$, with a similar argument in Theorem 6.3  \cite{LL}(can also be seen in the proof of Theorem 1  \cite{CL} or Proposition 5.2 \cite{MS2}), we can prove $u_0=w_0(x,0)$ is nonnegative.
\end{remark}

\section{\label{Deg} Proof of the main results}

This section is devoted to prove our  main results. Regarding $\lambda$ as a parameter and let $\lambda$ towards to infinity, we first prove the following proposition, which describes an important relation between $c_\lambda$ and $c(\Omega)$.

\begin{lemma}\label{lemma4.1}
  $c_{\lambda}\rightarrow c(\Omega)$ as $\lambda \rightarrow +\infty$.
\end{lemma}

\begin{proof}
  It is not difficult to find that $c_{\lambda}\leq c(\Omega)$ for each $\lambda\geq 0$.
We shall proceed through several claims on analyzing the convergence property of $c_\lambda$ as $\lambda\to +\infty$.

  \noindent{\bf Claim 1.}
 There exists $\Lambda_0>0$ for all $\lambda\geq \Lambda_0$ such that $c_{\lambda}$ is achieved by a $w_{\lambda}\in E_{\lambda}$
 \begin{proof}[\bf Proof of Claim 1.]
  Since $c_{\lambda}\leq c(\Omega)$, then it follows from Proposition \ref{pro2.1} that there exists a $\Lambda_0>0$ such that for any $\lambda>\Lambda_0$, $c_\lambda$ is achieved by a critical point $w_\lambda\in E_\lambda$ of $J_{\lambda}$.
\end{proof}
    Let $\lambda_n \to \infty$, from the above commentaries, for each $\lambda_n$ there exists a  $w_{\lambda_n}\in E_{\lambda_n}$ with $J_{\lambda_n}(w_n)=c_{\lambda_n}$ and $J'_{\lambda_n}(w_n)=0$.
    With a similar argument as \eqref{eqs2.7}, we have  $\|w_n\|_{\lambda_n}\leq 4c_{\lambda_n}\leq 4c(\Omega)$. By using Lemma \ref{lemma2.2}, it yields that $\{w_n\}$ is bounded in $E$ for $n$ large enough. Hence, there exists a $w\in E$ such that $w_n\rightharpoonup w$ in $E$ up to a subsequence  and
  \begin{equation}\label{eqs4.1}
    w_{n}(x,0)\rightarrow w(x,0) ~~\text{in}~~ L_{loc}^{q}(\R^N)~~ \text{for} ~~ 2\leq q <2_{s}^{*}.
  \end{equation}

  \noindent{\bf Claim 2.}
   $w(x,0)=0$ a.e. in $\Omega^c$, where $\Omega^c=\{x\in \R^N~|~ x\notin \Omega\}$, hence $w(x,0)\in E_0$.
\begin{proof}[\bf Proof of Claim 2.]
  Since   $f(t)t\geq 2F(t)\geq 0$ for $t\in \R$,  we then have
\begin{align*}
 &J_{\lambda_n}(w_n)-\frac{1}{4}\langle J'_{\lambda_n}(w_n),w_n \rangle \\
=&\frac{1}{4}\int_{\R_{+}^{N+1}}y^{1-2s}|\nabla w_n|^{2}dxdy+\frac{\lambda_n}{4}\int_{\R^N}V(x)|w_{n}(x,0)|^{2}dx\\
&+\frac{1}{4}\int_{\R^N}\Big(I_{\alpha}\ast F(w_n(x,0)) \Big)\Big(2f(w_n(x,0))w_n(x,0)-F(w_n(x,0))\Big)dx \\
\geq &\frac{\lambda_n}{4}\int_{\R^N}V(x)|w_{n}(x,0)|^{2}dx.
\end{align*}
From the analysis above, we can conclude that
\begin{equation}\label{hu-a1}
 \int_{\R^N}V(x)|w_{n}(x,0)|^{2}dx \leq \frac{4c_{\lambda_n}}{\lambda_n}.
\end{equation}
 By Fatou's Lemma,
 $$
 \int_{\Omega^c}V(x)|w(x,0)|^2dx \leq \liminf_{n\rightarrow +\infty}\int_{\R^N}V(x)|w_{n}(x,0)|^{2}dx \leq 0,
 $$
 which implies that $V(x)|w(x,0)|^2=0$ a.e. in ${\Omega^c}$. Note that, by condition $(V_1)$, $V(x)\ne0$ a.e. in ${\Omega^c}$. Thus we have $w(x,0)=0$ a.e. in  ${\Omega^c}$.
\end{proof}

\noindent{\bf Claim 3.}
 $w_n(x,0)\rightarrow w(x,0)$ strongly in $L^{q}(\R^N)$ for $2\leq q < 2_{s}^{*}$.
\begin{proof}[\bf Proof of Claim 3.]
Set $v_n=w_n-w$. We first assert that, for a fixed $r>0$,
\begin{equation*}
\widetilde{\delta}=\liminf_{n\rightarrow +\infty}\sup_{y\in \R^N}\int_{B_{r}(y)}|v_n(x,0)|^{2}dx >0.
\end{equation*}
Assume by contradiction that there exists a sequence $\{x_n\}\subset\R^N$ satisfies $|x_n|\to\infty$ and
 $$
 \int_{B_{r}(x_n)}|v_{n}(x,0)|^{2}dx\geq\widetilde{ \delta}/2 >0,\quad \text{for $n$ large enough.}
 $$
 Similar to \eqref{hu-a1}, one hand we have
 \begin{equation*}
    \lim_{n\to\infty}\int_{\R^N}V(x)|w_{n}(x,0)|^{2}dx =0.
 \end{equation*}
 On the other hand,
\begin{equation}\label{hu-a2}
\begin{split}
\int_{\R^N}V(x)|w_{n}(x,0)|^{2}dx
   &\geq \int_{B_{r}(x_n)\cap \{x|V(x)> M\}}V(x)|w_{n}(x,0)|^{2}dx  \\
   &=\int_{B_{r}(x_n)\cap \{x|V(x)> M\}}V(x)|v_{n}(x,0)|^{2}dx \\
   &\geq M\left(\int_{B_{r}(x_n)}|v_n(x,0)|^{2}dx-\int_{B_{r}(x_n)\cap \{x|V(x)\leq M\}}|v_{n}(x,0)|^{2}dx\right)\\
   &\geq \frac{M \widetilde{\delta}}{2}-o_n(1),
\end{split}
\end{equation}
where in the last inequality, we use the assumption $(V_2)$, that is $\mu\big( B_{r}(x_n)\cap \{x\mid V(x)\leq M\}\big)\rightarrow 0$ as $n\rightarrow \infty$, and the boundedness of $\{v_n\}$ in $E$.
Taking the limit $n\to\infty$ in \eqref{hu-a2}, we get a contradiction, hence $\widetilde{\delta}=0$ holds. Then by the Concentration Compactness Lemma \cite{pl}, we obtain $w_n(x,0)\rightarrow w(x,0)$  strongly in $L^{q}(\R^N)$ for $2\leq q< 2_{s}^{*}$.
\end{proof}
\noindent{\bf Completion of the Proof of Lemma 4.1:}
By  Claim 3, we then can easily prove $w$ is a  weak solution to the following problem
\begin{equation}\label{eqs4.2}
  \begin{cases}
    -div (y^{1-2s}\nabla w)=0 ~~\text{in}~~\R_{+}^{N+1}, \\
    \frac{\partial w}{\partial \nu}= \displaystyle\int_{\Omega}\frac{ F(w(z))}{|x-z|^{N-\alpha}}dzf(w) ~~ \text{on}~~ \Omega\times \{0\}, \\
    w=0 ~~\text{on}~~ \R^{N} \setminus \Omega \times \{0\}.
  \end{cases}
\end{equation}
Hence $w$ belongs to $\mathcal{N}_{0}$.
Furthermore, by using Hardy-Littlewood-Sobolev inequality and $w_n(x,0)\rightarrow w(x,0)$ strongly in $L^{q}(\R^N)$ for $2\leq q<2_{s}^{*}$ again, we get
\begin{align*}
 J_{\lambda_n}(w_n) &=\frac{1}{2} \int_{\R^N}\Big(I_{\alpha}\ast F(w_n(x,0))\Big) \Big(f(w_n(x,0))w_n(x,0)-F(w_n(x,0))\Big)dx \\
  &=\frac{1}{2} \int_{\Omega}\int_{\Omega}\frac{ F(w(x,0))\Big(f(w(z,0))w(z,0)-F(w(z,0))\Big)}{|x-z|^{N-\alpha}}dx dz +o_{n}(1)\\
  &=\frac{1}{2} \int_{\R_{+}^{N+1}}y^{1-2s}|\nabla w|^{2}dxdy-\frac{1}{2}\int_{\Omega}\int_{\Omega}\frac{F(w(x,0))F(w(z,0))}{|x-z|^{N-\alpha}}dxdz +o_{n}(1) \\
  &=J_0(w) +o_{n}(1).
\end{align*}
 Then $\displaystyle c(\Omega)\leq J_{0}(w)= \lim_{n\to \infty}J_{\lambda_n}(w_n)$ as $n\to \infty$, from which combining with conclusion $c_{\lambda_n}\leq c(\Omega)$, we complete the proof.
\end{proof}

Next we  give the proofs of Theorems \ref{th1.1} and \ref{th1.2}.
\begin{proof}[\bf Proof of Theorem \ref{th1.1}]
For $\lambda_n$ big enough we suppose that $c_{\lambda_n}$ is achieved by a critical point $w_{\lambda_n}\in E_{\lambda_n}$ of $J_{\lambda_n}$, i.e.
$J_{\lambda_n}(w_{\lambda_n})= c_{\lambda_n}$ and $J'_{\lambda_n}(w_{\lambda_n})= 0$. Let $u_{\lambda_n}=:w_{\lambda_n}(x,0)$.
The main result of Theorem \ref{th1.1} is to prove $\{u_{\lambda_n}\}$ converges to a least energy solution to equation \eqref{eqs1.7} in $E$ up to a subsequence as $\lambda_n \to \infty$.

With a similar argument in the proof of Lemma \ref{lemma4.1},
we can prove $\{w_{\lambda_n}\}$ is bounded in $E$, and there exists a $w_0\in E$ such that $w_{\lambda_n}\rightharpoonup w_0$ in $E$. Moreover,  $w_{\lambda_n}(x,0)\rightarrow w_0(x,0)$ strongly in $L^{q}(\R^N)$ for $q\in [2,2_{s}^{*})$. Thus, $w_0$ solves equation \eqref{eqs1.8} and $J_0(w_0)=c(\Omega)$.
Before closing the proof, we still need to prove that $w_n\to w_0$ strongly in $E$.
By calculation, we have
\begin{multline}\label{eqs4.3}
\|w_{\lambda_n}-w_0\|_{\lambda_n}^{2}=
\displaystyle\langle J'_{\lambda_n}(w_{\lambda_n}),w_{\lambda_n}\rangle-\langle J'_{\lambda_n}(w_{\lambda_n}), w_0\rangle\\
 +\int_{\R^N}\Big(I_{\alpha}\ast F(w_{\lambda_n}(x,0))\Big)f(w_{\lambda_n}(x,0))\big(w_{\lambda_n}(x,0)-w_0(x,0)\big)dx + o_{n}(1).
\end{multline}
Applying the Hardy-Littlewood-Sobolev inequality and $w_{\lambda_n}(x,0)\to w_0(x,0)$ strongly in $L^{q}(\R^N)$ for $q\in [2,2_{s}^{*})$, we deduce that
\begin{equation*}
\|w_{\lambda_n}-w_0\|_{\lambda_n}^{2}\to0, \;\;\text{as $n\to\infty$}.
\end{equation*}
Then, by Lemma \ref{lemma2.2}, we have $w_{\lambda_n}\to w_0$ strongly in $E$, furthermore $u_{\lambda_n}\to u_0:= w_{0}(x,0)$ strongly in $H^s(\R^N)$.
The proof is completed.
\end{proof}

\begin{proof}[\bf Proof of Theorem \ref{th1.2}]

Suppose $\{w_{\lambda_n} \}\subset H^{s}(\R^N)$  is a sequence of solutions to equation \eqref{eqs1.6}
with $\lambda$ being replaced by $\lambda_n$ and  $J_{\lambda_n}(w_{\lambda_n})<\infty$ as $\lambda_n \to \infty$, then we konw that $u_{\lambda_n} = w_{\lambda_n}(x, 0)$ satisfies equation \eqref{eqs1.1}. It is easy to see that  $\{w_{\lambda_n}\}$ must be bounded in $E$. We may assume that $w_{\lambda_n}\rightharpoonup w$ weakly in $E$ and
$w_{\lambda_n}(x,0)\rightarrow w(x,0)$ strongly in $L_{loc}^{q}(\R^N)$ for $q\in [2,2_{s}^{*})$.  Same as  the proof of Lemma \ref{lemma4.1}, we can
prove that $w|_{\Omega^c}= 0$ and $w\in E_0$ is solution to \eqref{eqs1.8}. Moreover $w_{\lambda_n}(x,0)\to w(x,0)$ strongly in $L^{q}(\R^N)$ for $q\in [2,2_{s}^{*})$. With a similar argument in the proof of Theorem \ref{th1.1},  we only need to prove $w_{\lambda_n} \to w$ strongly in $E$.
\begin{align*}
 &\int_{\R_{+}^{N+1}}y^{1-2s}|\nabla w_{\lambda_n}-\nabla w|^{2}dxdy + \int_{\R^N}\lambda_n V(x)|w_{\lambda_n}(x,0)-w(x,0)|^{2}dx \\
  &=\int_{\R_{+}^{N+1}}y^{1-2s}|\nabla w_{\lambda_n}|^{2}dxdy+\int_{\R^N}\lambda_n V(x)|w_{\lambda_n}(x,0)|^{2}dx \\
  & \ \ \ \ \ \
  -\int_{\R_{+}^{N+1}}y^{1-2s}|\nabla w|^{2}dxdy-\int_{\R^N}\lambda_n V(x)|w(x,0)|^{2}dx +o_{n}(1) \\
  &= \int_{\R^N}\Big(I_{\alpha}\ast F(w_{\lambda_n}(x,0))\Big)f(w_{\lambda_n}(x,0))w_{\lambda_n}(x,0)dx -\int_{\Omega}\int_{\Omega}\frac{F(w(x,0))f(w(z,0))w(z,0)}{|x-z|^{N-\alpha}}dxdz +o_{n}(1) \\
  & \to 0 \ \ \text{as}~~ n \to \infty.
\end{align*}
Thus we have $w_{\lambda_n} \to w $ strongly in $E$ and  complete the proof.

\end{proof}

\noindent {\bf Acknowledgements:}
The authors would like to thank Prof. Shuangjie Peng  very much for helpful suggestions on the present paper.

\end{document}